\documentclass[preprint,12pt]{elsarticle}
\usepackage{amsmath,color}
\usepackage{amssymb}
\usepackage{amsthm}
\usepackage{graphicx}
\usepackage{color}
\usepackage{verbatim}
\usepackage[pdfencoding=auto]{hyperref}

\usepackage{tikz}
\usetikzlibrary{shapes,automata,positioning,decorations.pathmorphing}

\usepackage{lmodern}
\usepackage{url}

\newtheorem{theorem}{Theorem}
\newtheorem{corollary}{Corollary}
\newtheorem{lemma}{Lemma}
\newtheorem{proposition}{Proposition}

\journal{Journal of Computational and Applied Mathematics}

\newcommand{\gf}{generating function}
\newcommand{\bg}{\mathbf{g}}
\newcommand{\M}{\mathbf{M}}

\newcommand{\w}{w}
\newcommand{\n}{\ell}

\newcommand{\X}{X}
\newcommand{\PPP}{\mathbb{P}}

\newcommand{\PP}{P}
\newcommand{\F}{F}

\newcommand{\A}{A}
\newcommand{\Fiid}{\widetilde{\F}}
\newcommand{\Giid}{\widetilde{G}}

\newcommand{\FS}{Q}                     

\newcommand{\B}{U}

\newcommand{\be}{\mathbf{e}}

\newcommand{\pd}{\partial}

\newcommand*{\colorboxed}{}
\def\colorboxed#1#{%
  \colorboxedAux{#1}%
}
\newcommand*{\colorboxedAux}[3]{%
  \begingroup
    \colorlet{cb@saved}{.}%
    \color#1{#2}%
    \boxed{%
      \color{cb@saved}%
      #3%
    }%
  \endgroup
}

\begin{document}
\begin{frontmatter}

\title{Multiple consecutive runs of multi-state trials: distributions of \texorpdfstring{$(k_1, k_2, \dots, k_\n)$}{k1, k2, ..., kl} patterns}



\author{Yong Kong}
\address{
Department of Biostatistics
School of Public Health,
Yale University,
300 Cedar Street, New Haven, CT 06520, USA
email: \texttt{yong.kong@yale.edu} }

\newpage

\begin{abstract}
The pattern $(k_1, k_2, \dots, k_\n)$ is defined to have
  at least  $k_1$ consecutive $1$'s followed by at least  $k_2$ consecutive $2$'s,
  $\dots$, followed by at least  $k_\n$ consecutive $\n$'s.
  By iteratively applying the method that was developed previously to
  decouple the combinatorial complexity involved in studying complicated patterns
  in random sequences,
  the distribution of pattern $(k_1, k_2, \dots, k_\n)$ is derived
  for arbitrary $\n$.
  Numerical examples are provided to illustrate the results.
\end{abstract}


\end{frontmatter}

\section{Introduction} \label{S:intro}

For a multi-state trial with $\n$ outcome states,
with the states labeled as the numbers in the set $\{1, 2,, \dots, \n\}$,
the pattern $(k_1, k_2, \dots, k_\n)$ is defined to have
at least  $k_1$ consecutive $1$'s followed by at least  $k_2$ consecutive $2$'s,
$\dots$, followed by at least  $k_\n$ consecutive $\n$'s,
where $k_i$, $i=1, 2, \dots, \n$ are positive numbers.

As an example, for $\n=3$, $k_1=2$, $k_2=1$, and $k_3=2$,  the following sequence contains $3$
$(k_1, k_2, k_3)$ patterns:
\begin{equation} \label{E:example}
  223\colorboxed{red}{\!11123333\!}112\colorboxed{red}{\!1122233\!}2233\colorboxed{red}{\!112333\!}12 .
\end{equation}

The $(k_1, k_2, \dots, k_\n)$ pattern is a generalization of $(k_1, k_2)$ pattern.
The distribution of $(k_1, k_2)$ was investigated by different researchers
\cite{HUANG1991125,DAFNIS20101691,UPADHYE201819}.
\citet{HUANG1991125} named the distribution of $(k_1, k_2)$
as a binomial distribution of order $(k_1, k_2)$,
which was in turn motivated by the many studies on the well-known
binomial distribution of order $k$
\cite{Feller1968}.
The binomial distribution of order $k$ has been discussed extensively in literature
and
\citet{Balakrishnan2002} gives an excellent summary of the past and recent results of this and related distributions.

The study of the distributions of certain patterns in random sequences,
while itself an interesting mathematical problem,
has found many applications in a wide range of fields
~\citep{Balakrishnan2002,Knuth1997b,glaz2019handbook},
from biology and genomics research, to computer algorithms, to reliability analysis.
For the particular distribution of $(k_1, k_2)$ pattern,
~\citet*{DAFNIS20101691} gave an interesting application in meteorology and agriculture
where cultivation of raisins in Greece is considered.
For the ingathering a period of at least $4$ consecutive dry days is required. Before that period starts, a period of
at least $2$ consecutive rainy days is required in order for the raisins to be watered.
Here we have $k_1=2$ and $k_2=4$.
The distribution of $(k_1, k_2)$ was also used in the field of machine maintenance~\cite{UPADHYE201819}
where two types of machines were considered.
Related applications can be found in start-up demonstration tests and shock models
\cite{Balakrishnan2002,Balakrishnan2014}.
For applications in biology, it is known that
genes are often flanked by several binding sites for distinct transcription factors,
and efficient expression of these genes requires the cooperative actions of several different transcription factors.

The distribution of $(k_1, k_2)$ can only consider two outcome states.
However, quite often the outcomes have more than two states in practice.
In order for the theory to be applicable in these multiple-state situations,
for example in a new multi-state start-up
demonstration test where each start-up trial is
divided into more than two states,
the theory of $(k_1, k_2)$ distribution needs to be generalized from $\n=2$ to $\n > 2$,
to a theory of $(k_1, k_2, \dots, k_\n)$ distribution.
To the best of the author's knowledge, no theory of $(k_1, k_2, \dots, k_\n)$ distribution with
$\n > 2$ exists in the literature,
let alone a theory of $(k_1, k_2, \dots, k_\n)$ distribution with arbitrary $\n$.

Study of runs and related patterns in random sequences has a long history,
dated back to the beginning of the probability theory (see \cite{Mood1940,Feller1968,Glaz2001,Balakrishnan2002},
and references therein).
As the patterns under study become complicated, especially when dealing with multi-state or multiset systems,
the combinatorial complexity involved may become more challenging.
Several systematic methods have been devised to overcome the combinatorial
difficulties,
notably among them is the versatile finite Markov chain imbedding approach
\citep{Fu1994,Koutras1995,fu2003}. 
Another elegant systematic approach is the method of~\citet{Koutras1997b},
which takes advantage of the fact that the waiting time of a pattern
is usually easier to calculate than the number of patterns.
It derives the double generating function of the
number of appearances of a pattern by using
the generating function of the waiting time for the $r$th
appearance of the pattern.

We have introduced another systematic method to study various
distributions of runs in systems composed of multiple kinds of objects
that is inspired by methods in statistical physics
~\citep{Kong2006,Kong2016iwap,Kong2019}.
The method uses a two-step approach to decouple the study of complicated patterns
into two easy independent steps.
In the first step each kind of objects under study
(such as the ``success'' and ``failure'' states in a Bernoulli trial,
or different letter types in a multivariate random sequence)
is dealt with in isolation 
without considerations of the other kinds of objects.
In this step each object is encapsulated by a generating function $g_i$.
Since there is no consideration of other objects in this step,
the combinatorial complexities are eliminated and $g_i$ can be easily obtained.
In the second step, these individual generating functions $g_i$'s
are combined to form a generating function $G$ for the whole system.

This method has been applied to investigate various runs and patterns in random
sequences~\citep{EryIlmaz2008,eryilmaz2008run,Kong2014,Kong2016,Kong2015,Kong2018,Kong2020}.
In these previous works, the results were obtained by applying the method once directly.
The method, however, has the ability to  be applied iteratively
to build more complicated patterns from simpler patterns~\cite{,Kong2016iwap,Kong2019}.
We have briefly mentioned this possibility before~\cite{,Kong2016iwap,Kong2019},
but have not applied the method in this iterative fashion to a concrete pattern. 
In this paper this idea will be used to study patterns in multi-state trials.
In particular, the distribution of $(k_1, k_2, \dots, k_\n)$ pattern
will be derived by applying the method iteratively.

The rest of the paper is organized as follows.
In Section~\ref{S:method} a brief introduction of the general method is described.
Since the method is relatively new, this section is included here for the benefits of those readers
who are not familiar with the method.
In Section~\ref{S:2} we first obtain the distribution for $\n=2$.
Then using a variation of the results of $\n=2$ (Section~\ref{S:2p}),
we derive the  distribution for $\n=3$ (Section~\ref{S:3}).
The complete solution of the general case for arbitrary $\n$ is derived in Section~\ref{S:gen}.
The major result of the paper is Theorem~\ref{Th:genp},
where the generating function of distribution of pattern $(k_1, k_2, \dots, k_\n)$
for arbitrary $\n$ is given in a closed form.

\subsection{Notation} \label{SS:notation}
For positive number $\n  \geq 2$, let the states be labeled with the numbers in the set $\{1, 2,, \dots, \n\}$.
Let $p_i$ denote the probability of the $i$th state to appear in each position of the trial,
with $\sum_{i=1}^\n p_i = 1$.
Let $\X_{\n; k_1,k_2, \dots, k_\n}^{(n)}$ denote the number of $(k_1, k_2, \dots, k_\n)$ patterns in a sequence of length $n$.
In the particular example sequence~\eqref{E:example} shown at the beginning of this section,
we have $\X_{3; 2,1,2}^{(33)} = 3$.

Let $\PP_{\n,n}( m  )$ denote the probability of sequences of length $n$ with $m$ $(k_1, \dots, k_\n)$ patterns:
$\PP_{\n,n}( m  ) = \PPP(\X_{\n; k_1,k_2, \dots, k_\n}^{(n)} = m)$.
The sum of $\PP_{\n,n}( m  )$ over all possible $m$ is $1$: $\sum_{j=0} \PP_{\n, n}(j) = 1$.
The probability \gf{} of $\PP_{\n, n}( m )$ for sequences of length $n$ is denoted by
\[
  \F_{\n,n}(\w) = \sum_{j=0} \PP_{\n, n}(j) \w^j,
\]
and the double probability generating function is denoted by
\[
  G_\n(\w, z) = \sum_{n=0}^\infty \F_{\n,n}(\w) z^n =    \sum_{n=0}^\infty \sum_{j=0} \PP_{\n, n}(j) \w^j z^n  .
\]
Similarly, for  independent identical distribution (i.i.d.) sequences,
where $p_1 = p_2 = \cdots = p_\n =  1/\n$,
let $\A_{\n,n}(m)$
denote the \emph{number} of sequences
that contain $m$ $(k_1, \dots, k_\n)$ patterns for trial sequences of length $n$.
The sum of $\A_{\n,n}(m)$ over all possible $m$ is $\n^n$: $\sum_{j=0} \A_{\n, n}(j) = \n^n $.
The \gf{} of $\A_{\n,n}(m)$ for sequences of length $n$ is denoted by
\[
  \Fiid_{\n,n}(\w) = \sum_{j=0} \A_{\n, n}(j) \w^j,
\]
and the double probability generating function is denoted by
\[
  \Giid_\n(\w, z) = \sum_{n=0}^\infty \Fiid_{\n,n}(\w) z^n =    \sum_{n=0}^\infty \sum_{j=0} \A_{\n, n}(j) \w^j z^n  .
\]

In the following $[x^n]f(x)$ is used to denote
the coefficient of $x^n$ in the series of $f(x)$ in powers of $x$.

\section{The general method: decoupling the combinatorial complexity
into two simple independent steps} \label{S:method}

For an abstract pattern under study,
the general method~\citep{Kong2006,Kong2016iwap,Kong2019} for decoupling the combinatorial complexity
into two simple independent steps
consists of two key elements, the blocks of sequences and the ``interactions'' between
adjacent blocks.
Suppose we have $\n$ kinds of objects, such as the numbers in the set $\{1, 2,, \dots, \n\}$.
Each object can appear multiple times.
When these objects form a linear sequence, we can divide the sequence into blocks, based on the patterns we're investigating.
Each block can be described by a generating function (denoted by $g_i$'s below)
that encapsulates the objects that make up the block.
Each block represents a particular pattern of the objects, such as a run of object $i$ with length less than $k_i$.
More generally, the blocks can also represent compound patterns.
A compound pattern is a pattern composed of a set of finer-grained patterns.
It's this feature that makes the method powerful to attack complicated patterns in an iterative way,
to build up a whole system from simpler sub-systems
(see the applications to the $(k_1, k_2, \dots, k_\n)$ pattern
in Section~\ref{S:3} and Section~\ref{S:gen}).

The ``interactions'' (denoted by $\w_{ij}$'s below) between the blocks are used to create restrictions on the relations between
adjacent blocks, and they can be arbitrary.
The term ``interactions'' is taken from the study of cooperativity
in biochemistry and biophysics~\citep{di1996theory}.

The whole system can be combined and enumerated
by the following matrix formula,
resulting in a generating function $G$ of the whole system
\citep{Kong2006,Kong2016iwap,Kong2019}:

\begin{equation} \label{E:matrix}
  G = \be \M^{-1} \bg,
\end{equation}
where
$\M$ is a $s \times s$ matrix,
\begin{equation*} \label{E:ek}
   \be = [\underbrace{1, 1, \cdots, 1}_s],
\end{equation*}
\begin{equation} \label{E:g}
   \bg = [g_1, g_2, \cdots, g_s]^T,
\end{equation}
and
\begin{equation} \label{E:M}
 \M =   
  \begin{bmatrix}
           1 & -g_1 \w_{12}  & -g_1 \w_{13}  & \cdots & -g_1 \w_{1s} \\
 -g_2 \w_{21} &  1           & -g_2 \w_{23}  & \cdots & -g_2 \w_{2s} \\
      \vdots & \vdots       & \vdots       & \vdots & \vdots      \\
 -g_s \w_{s1} & -g_s \w_{s2}  & -g_s \w_{s3}  & \cdots & 1 
   \end{bmatrix}.
\end{equation}
Here $g_i$ is the generating function of
the $i$th individual object
when considered in isolation (without considering the interactions with the other objects),
and $\w_{ij}$ represents the interactions between an 
$i$th kind of block \emph{on the left} and a $j$th kind block \emph{on the right}.
The  $\w_{ij}$'s can be used to mark the interactions between blocks represented by $g_i$ and $g_j$.
When $\w_{ij}$ is set to $0$, the interactions between blocks $g_i$ and $g_j$ are forbidden.
In other words, no configurations of $g_i \cdot g_j$ is allowed in the sequences.

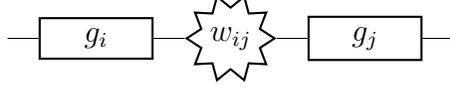
\begin{center}
\begin{tikzpicture}[] 
  \begin{scope}[every node/.style={thick}]
    \node[draw, star,star points=10,star point ratio=0.7, minimum width=0.7cm] (w) at (2,0) {$\w_{ij}$};
    \node[draw, rectangle, left=1em,  minimum width=1.5cm]  at (w.west) (i)  {$g_i$};
    \node[draw, rectangle, right=1em, minimum width=1.5cm] at (w.east) (j) {$g_j$};
    \node[left=1em of i] (ii) {};
    \node[right=1em of j] (jj) {};
    \draw [-]  (i) edge (w);
    \draw [-]  (w) edge (j);
    \draw [-]  (ii) edge (i);
    \draw [-]  (j) edge (jj);
\end{scope}
\end{tikzpicture}
\end{center}

The method can be used iteratively to study more complicated patterns.
The output from the method, the generating function $G$ built up from the lower level generating function $g_i$'s,
can be used as inputs in Eq.~\eqref{E:g} to feed into Eq.~\eqref{E:matrix} again,
to obtain the next level of generating functions.
Suppose in the $t$th iteration, for the $j$th sub-system,
the input generating functions in Eq.~\eqref{E:g} are labeled as
$g_{j, i}^{(t)}$, $i, j=1,2, \dots$.  With these $g_{j, i}^{(t)}$ as inputs,
we can use Eq.~\eqref{E:matrix} to compute a new set of generating functions of the next level (the $(t+1)$th iteration):
\[
  g_j^{(t+1)} (  g_{j, 1}^{(t)}, g_{j, 2}^{(t)}, \cdots), \quad j=1, 2, \dots
\]
Those $g_j^{(t+1)}$ can then in turn be used in the next iteration as inputs to obtain the $(t+2)$th set of generating functions
\[	
  g_j^{(t+2)} \left[g_1^{(t+1)} (  g_{1, 1}^{(t)}, g_{1, 2}^{(t)}, \cdots), g_2^{(t+1)} (  g_{2, 1}^{(t)}, g_{2, 2}^{(t)}, \cdots), \cdots
    \right].
\]
We will illustrate in the following this iterative process of building up more complicated system from
relatively simpler sub-systems by applying the method to the $(k_1, k_2, \dots, k_\n)$ patterns.

\section{The case of \texorpdfstring{$\n=2$}{l=2}} \label{S:2}

First, we consider the case for $\n=2$.
The following Proposition gives the double \gf{} for $\n=2$.

\begin{proposition}[distribution of $(k_1, k_2)$] \label{P:k1k2p}
  The double generating function for $\PP_{2, n} ( m )$ is given by
\begin{equation} \label{E:k1k2p}
  G_2(\w, z) = \frac{1}{1 - z - (w-1) p_1^{k_1}p_2^{k_2} z^{k_1+k_2}} .
\end{equation}

\begin{proof}
  Let the matrix $\M$ in Eq.~\eqref{E:M} be the following $4 \times 4$ matrix
  \[
    \M = 
  \begin{bmatrix}
    1 & 0 & -g_{1}' & -g_{1}' \\
    0 & 1 & -g_{1} & -g_{1} \w \\
    -g_{2}' & -g_{2}' & 1 & 0 \\
    -g_{2} & -g_{2} & 0 & 1
  \end{bmatrix},
\]
where we represent each state by two \gf{}s: $g_i$ and $g_i'$ $(i=1, 2)$.
The $g_i$ and $g_i'$ $(i=1, 2)$ are given by
\begin{equation} \label{E:gi}
\begin{aligned}
  g_i' &= \sum_{j=1}^{k_i-1} (p_i z)^j = \frac{p_i z - (p_i z)^{k_i} }{1 - p_i z}, \\
  g_i  &= \sum_{j=k_i}^\infty  (p_i z)^j = \frac{ (p_i z)^{k_i} } {1 - p_i z} .
\end{aligned}
\end{equation}
and
\[
     \bg = [g_1', g_1, g_2', g_2]^T.
   \]

For $i=1, 2$,
the generating function $g_{i}'$ is used to represent runs of $i$'s with a length less than $k_i$,
  while the generating function $g_{i}$ is used to represent runs of $i$'s with a length
  at least $k_i$.

The interaction factor $\w_{2, 4} = \w$ is used to track the number of patterns with a block of $g_{1}$ on the left and
a block of $g_{2}$ on the right:
its row index in the matrix $\M$ is for $g_1$,
while its column index in the matrix $\M$ is for $g_2$,
hence it represents configurations $g_1 \cdot g_2$.

  The interaction factors between $g_{1}'$ and $g_{1}$ are zero: $\w_{1, 2} = \w_{2, 1} = 0$,
  since these two blocks are forbidden to be adjacent to each other.
  Otherwise there would be ambiguity for the configurations that are tracked by $\w$.
  Similarly, the interaction factors between $g_{2}'$ and $g_{2}$ are zero: $\w_{3, 4} = \w_{4, 3} = 0$.

  Substituting matrix $\M$ and the four generating functions into
  Eq.~\eqref{E:matrix} gives the generating function $G_2(\w, z)$ in Eq.~\eqref{E:k1k2p}.
\end{proof}
\end{proposition}
We mention here that Proposition~\ref{P:k1k2p} has also been obtained by
~\citet{HUANG1991125} and \citet*{DAFNIS20101691} via different methods.

\section{The case of \texorpdfstring{$\n=2$}{l=2}, with pattern \texorpdfstring{$(k_1, k_2)$}{(k1, k2)} at the right end} \label{S:2p}

Proposition~\ref{P:k1k2p} gives the distribution of a pattern of
at least $k_1$ consecutive $1$’s followed by at least $k_2$ consecutive $2$’s (the $(k_1, k_2)$ pattern)
for a complete system with $2^n$ sequences of $1$'s and $2$'s.
The pattern $(k_1, k_2)$ can appear anywhere in the sequences of length $n$.
To obtain the distribution of $(k_1, k_2, k_3)$ for $\n=3$,
with at least $k_1$ consecutive $1$’s followed by at least $k_2$ consecutive $2$’s
followed by at least $k_3$ consecutive $3$’s,
we need the information about the distribution of pattern $(k_1, k_2)$ \emph{at the right end} of the sequences,
so that an interaction factor can be used to link these blocks with blocks of at least $k_3$ consecutive $3$’s
on the right
to build up patterns of  $(k_1, k_2, k_3)$.
The following Lemma gives what we need.

\begin{lemma}[$\n=2$, with  pattern $(k_1, k_2)$ at the right end] \label{L:k1k2p_p}
  The generating function for the probability that
  the pattern $(k_1, k_2)$ appears at the right end of the sequences
  is given by
\begin{equation} \label{E:k1k2p_p}
  H_2(z)  = \frac{p_1^{k_1}p_2^{k_2} z^{k_1+k_2}}{ (1-p_2 z) (1-p_1 z-p_2z ) } .
\end{equation}
The complement generating function, for all other sequences that do not have such a pattern,
is given by
\begin{equation} \label{E:k1k2pp_p}
  H_2'(z)  = \frac{ z(1-p_2z)(p_1+p_2) - p_1^{k_1}p_2^{k_2} z^{k_1+k_2} }
  {(1-p_2z)(1 - (p_1+p_2)z) } .
\end{equation}
\begin{proof}
  We need an additional special block made of at least  $k_2$ consecutive $2$'s
  which only appears at the right end of the sequences;
  its interaction factors with blocks on its right
  are all zero. The following $5 \times 5$ matrix will be used:
  \[
    \M= 
  \begin{bmatrix}
    1 & 0 & -g_{1}' & -g_{1}' & -g_{1}' \\
    0 & 1 & -g_{1} & -g_{1} & -g_{1} u \\
    -g_{2}' & -g_{2}' & 1 & 0 & 0  \\
    -g_{2} & -g_{2} & 0 & 1 & 0  \\
    0 & 0 & 0 & 0 & 1
  \end{bmatrix},
\]
where $g_{i_j}$ are the same as those used in Proposition~\ref{P:k1k2p}, and
\[
     \bg = [g_1', g_1, g_2', g_2, g_2]^T.
   \]
   Here the interaction variable $u$ is used to track those blocks of at least $k_1$ consecutive $1$'s
   on the left (represented by $g_1$) followed by a block of at least $k_2$ consecutive $2$'s on the right
   (represented by the last $g_2$ in the vector $\bg$),
   with an additional requirement that the block occurs at the right end of the sequence.
   This requirement is fulfilled by the interaction factors on the last row of the matrix $\M$:
   these interaction factors are all zero:
   \[
     \w_{5,1} = \w_{5,2} = \w_{5,3} = \w_{5,4} = 0, 
   \]
   which guarantees that there is no block appears on the right of these
   special blocks, represented by the last $g_2$ in the vector of $\bg$.
   
Substituting matrix $\M$ and the four generating functions into
Eq.~\eqref{E:matrix} gives the generating function $\FS_2(u, z)$,
which includes two parts:
one part includes the pattern we are looking for (tracked by variable $u$) and the other part that does not.
If the generating function $\FS_2(u, z)$ is considered as a polynomial in $u$,
the second part is the constant term of this polynomial.
The first part is  what we need to build up the system of $\n=3$.
This part is obtained by extracting the coefficient of $u$ in $\FS_2(u,z)$:
\[
  H_2(z) = [u] \FS_2(u, z) =  \frac{p_1^{k_1}p_2^{k_2} z^{k_1+k_2}}{ (1-p_2 z) (1-p_1 z-p_2z ) }  .
\]
The complement generating function Eq.~\eqref{E:k1k2pp_p} is obtained by subtracting $H_2(z)$
from the full system of $\n=2$:
\[
  H_2'(z) = \frac{ (p_1+p_2)z }{ 1 - (p_1+p_2)z  } - H_2(z)
  = \frac{ z(1-p_2z)(p_1+p_2) - p_1^{k_1}p_2^{k_2} z^{k_1+k_2} }
  {(1-p_2z)(1 - (p_1+p_2)z) } .
\]
\end{proof}
\end{lemma}

\section{The case of \texorpdfstring{$\n=3$}{l=3}} \label{S:3}

We're now ready to use the results of $\n=2$ from Section~\ref{S:2p}
to obtain the distribution of pattern $(k_1, k_2, k_3)$ for $\n=3$.

\begin{proposition}[distribution of $(k_1, k_2, k_3)$] \label{P:k1k2k3p}
  The double probability generating function for $\n=3$ is given by
\begin{equation} \label{E:k1k2k3p}
  G_3(\w, z) = \frac{1 - p_2 z}{(1-z)(1 - p_2 z) - (\w-1)p_1^{k_1}p_2^{k_2}p_3^{k_3} z^{k}} ,
\end{equation}
where $k=k_1+k_2+k_3$.  
\begin{proof}
  Eq.~\eqref{E:k1k2k3p} is obtained by using the following matrix in Eq.~\eqref{E:matrix}
\[
  \begin{bmatrix}
    1 & 0 & -H_2' & -H_2' \\
    0 & 1 & -H_2 & -H_2 \w \\
    -g_{3}' & -g_{3}' & 1 & 0 \\
    -g_{3} & -g_{3} & 0 & 1
  \end{bmatrix},
\]
where $H_2$ and $H_2'$ are given in Lemma~\ref{L:k1k2p_p},
and
\[
     \bg = [H_2', H_2, g_3', g_3]^T,
\]
with $g_3'$ and $g_3$ given in Eq.~\eqref{E:gi} for $i=3$.

The interaction factors $w_{1,2}$ and $w_{2,1}$ are set to zero to
forbid the formation of configurations such as $H_2' \cdot H_2$ or
$H_2 \cdot H_2'$.
Similarly,
the interaction factors $w_{3,4}$ and $w_{4,3}$ are set to zero to
forbid the formation of configurations such as $g_3' \cdot g_3$ or
$g_3 \cdot g_3'$, as we did in Proposition~\ref{P:k1k2p} and Lemma~\ref{L:k1k2p_p}.

The interaction factor $\w_{2, 4} = \w$ is used to track configurations of
$H_2 \cdot g_3$.  By definitions of $H_2(z)$ and $g_3(z)$,
these configurations are exactly the patterns $(k_1, k_2, k_3)$
that we are looking for.

\end{proof}
\end{proposition}

A numerical example is given here to illustrate Eq.~\eqref{E:k1k2k3p}.
The recurrence relation in Corollary~\ref{C:recur}
discussed for the arbitrary $\n$ in Section~\ref{S:gen}
can be used for the calculation.
Alternatively, the \gf{}~\eqref{E:k1k2k3p} can be expanded into a series of $z$
by any modern symbolic mathematical package, such as Maple and Mathematica.
For $n=17$, $k_1=2, k_2=2, k_3=3$,
and $p_1=1/6$, $p_2=1/3$, $p_3=1/2$, we have from Eq.~\eqref{E:k1k2k3p}
the following probability generating function $\F_{3,17}(\w) = \sum_j \PP_{3,17} ( j ) \w^j$:
\[
  \F_{3,17}(\w) = 0.9939258642 + 0.006071881114 \w + 0.000002254704073 \w^2 .
\]
This probability generating function shows that for a short sequence
like this ($n=17$),
the probability to have one or two $(k_1, k_2, k_3)$ patterns is quite low
($\PP_{3,17}(1) = 0.006071881114$ and $\PP_{3,17}(2) = 0.000002254704073$, respectively),
and the probability that no $(k_1, k_2, k_3)$ pattern occurs is high ($\PP_{3,17}(0) = 0.9939258642$).

From Proposition~\ref{P:k1k2k3p} we can get the average and other statistics
of the distribution of $(k_1, k_2, k_3)$.
In the following Corollary the average of $\X_{3; k_1,k_2, k_3}^{(n)}$ is given in a closed form.
\begin{corollary}[average of $\X_{3;  k_1, k_2, k_3}^{(n)}$] \label{C:a3p}
  The average of $\X_{3; k_1, k_2, k_3}^{(n)}$ is given by
  \begin{equation} \label{E:a3p}
    E[ \X_{3; k_1, k_2, k_3}^{(n)}] = 
  p_1^{k_1}p_2^{k_2}p_3^{k_3}
  \left[
    \frac{n-k+1}{1-p_2}
    - \frac{p_2 (1 - p_2^{n-k+1}) }{(1-p_2)^2}
    \right],
  \end{equation}
  where $k=k_1+k_2+k_3$.
\begin{proof}
  Take the derivative of Eq.~\eqref{E:k1k2k3p} with respect to $\w$, and set $\w=1$:
\[
  \frac{ \pd G_3(z, \w)} {\pd w} \Big |_{w=1} = \frac{p_1^{k_1}p_2^{k_2}p_3^{k_3}  z^k}{(1 - p_2 z)(1 - z)^2} .
\]  
Eq.~\eqref{E:a3p} is obtained by extracting the coefficient of $z$ from the above equation.
\end{proof}
\end{corollary}

\section{The general case: the distribution of \texorpdfstring{$(k_1, \dots, k_\n)$}{k1,...,kl}} \label{S:gen}

In this section the distribution of $(k_1, \dots, k_\n)$
of arbitrary $\n$ will be derived, based on
the cases of small number $\n$ developed in the previous sections.

First, we will derive the distribution of
the pattern $(k_1, \dots, k_\n)$ at the right end of the sequence
in the following Lemma.
\begin{lemma}[distribution of pattern $(k_1, \dots, k_\n)$ at the right end of the sequence] \label{L:genp}
The generating function $H_\n(z)$ is given by
  \begin{equation} \label{E:H_genp}
  H_\n(z) = \frac{\prod_{i=1}^{\n} p_i^{k_i}  z^k }
  {(1 - z \sum_{i=1}^\n p_i  ) \prod_{i=2}^{\n} (1-p_iz)}, 
\end{equation}
where $k= \sum_{i=1}^\n k_i$ .
\begin{proof}
  We will prove Eq.~\eqref{E:H_genp} by induction on $\n$.
  The base case has been proved in Lemma~\ref{L:k1k2p_p} for $\n=2$.
  Assume Eq.~\eqref{E:H_genp} is true for $\n-1$. 
  By using the following matrix in Eq.~\eqref{E:matrix}
    \[
    \M = 
  \begin{bmatrix}
    1 & 0 & -H_{\n-1}' & -H_{\n-1}' & -H_{\n-1}' \\
    0 & 1 & -H_{\n-1}  & -H_{\n-1}  & -H_{\n-1} u \\
    -g_{\n}' & -g_{\n}' & 1 & 0 & 0  \\
    -g_{\n} & -g_{\n} & 0 & 1 & 0  \\
    0 & 0 & 0 & 0 & 1
  \end{bmatrix},
\]
and
\[
     \bg = [H_{\n-1}', H_{\n-1}, g_{\n}', g_{\n}, g_{\n}]^T,
\]
the generating function $\FS_\n(u, z)$ can be obtained.
Like $\FS_2(u, z)$ in Lemma~\ref{L:k1k2p_p},
$\FS_\n(u, z)$ consists of two parts:
the one which is tracked by the variable $u$ encapsulates the patterns we are looking for,
and the other which does not contain $u$ describes everything else.

Here as in Lemma~\ref{L:k1k2p_p}, the interaction factor $u$
is used to track configurations $H_{\n-1} \cdot g_{\n}$
(the second and fifth elements in $\bg$).
The interaction factors in the last row of $\M$ make sure that
these configurations $H_{\n-1} \cdot g_{\n}$ are at the right end of the sequences,
with no other blocks on the right of the blocks represented
by the last $g_{\n}$ in $\bg$:
\[
  \w_{5,1} = \w_{5,2} = \w_{5,3} = \w_{5,4} = 0.
\]

The generating function $H_\n(z)$ is obtained by extracting the coefficient of $u$ of $\FS_\n(u, z)$ as
\[
  H_\n(z) = [u] \FS_\n(u, z) = \frac{ g_\n H_{\n-1}(z) (1+g_\n'+g_\n)}
  {1-(g_\n' + g_\n)( H_{\n-1}'(z) + H_{\n-1}(z)  )} .
\]
Eq.~\eqref{E:H_genp} is proved by substituting into the above equation the following identities:
\begin{align*}
  g_\n &= \sum_{j=k_\n}^\infty  (p_\n z)^j = \frac{ (p_\n z)^{k_\n}} {1 - p_\n z}, \\
  g_\n + g_\n' &= \frac{ p_\n z} {1 - p_\n z}, \\
  H_{\n-1}'(z) + H_{\n-1}(z) &= \frac{ (p_1 + p_2 + \cdots + p_{\n-1}) z}{1 - (p_1 + p_2 + \cdots + p_{\n-1}) z}.
\end{align*}
The last identity of the sum of $H_{\n-1}(z)$ and $H_{\n-1}'(z)$ is from the definition of the two complementary
generating functions.
\end{proof}
\end{lemma}

The i.i.d. version of Lemma~\ref{L:genp},
for independent identical distribution multi-state trials (see Section~\ref{S:iid}),
can be derived from Eq.~\eqref{E:H_genp} as
\begin{corollary}[for i.i.d. distribution of pattern $(k_1, \dots, k_\n)$ at the right end of the sequence]
The generating function $\widetilde{H}_\n(z)$ for i.i.d. trials is given by  
\[
  \widetilde{H}_\n(z) = \frac{z^k}
  {(1-z)^{\n-1} (1-\n z)},
\]
where $k= \sum_{i=1}^\n k_i$ .
\end{corollary}
  \begin{proof}
    Let $p_i = 1/\n$ for $i=1, 2, \dots, \n$ and substitute $z$ by $\n z$ in Eq.~\eqref{E:H_genp}.
    \end{proof}

By using Lemma~\ref{L:genp}, the distribution of $(k_1, \dots, k_\n)$ for the arbitrary $\n$ can be obtained,
    which is stated in the following Theorem.
\begin{theorem}[distribution of $(k_1, \dots, k_\n)$] \label{Th:genp}
  The double generating function for the distribution of $(k_1, \dots, k_\n)$
  is given by
\begin{equation} \label{E:genp}
  G_\n(\w, z) = \frac{ \prod_{i=2}^{\n-1} (1-p_i z)    }
  {(1-z) \prod_{i=2}^{\n-1} (1-p_iz)   - (w-1) \prod_{i=1}^{\n} p_i^{k_i}   z^{k}},
\end{equation}
where $k= \sum_{i=1}^\n k_i$.
\begin{proof}
  By using the following matrix in Eq.~\eqref{E:matrix}
    \[
    \M = 
  \begin{bmatrix}
    1 & 0 & -H_{\n-1}' & -H_{\n-1}' \\
    0 & 1 & -H_{\n-1}  & -H_{\n-1} \w \\
    -g_\n' & -g_\n' & 1 & 0 \\
    -g_\n & -g_\n & 0 & 1
  \end{bmatrix},
\]
and
\[
  \bg = [H_{\n-1}', H_{\n-1}, g_{\n}', g_{\n}]^T,
\]
we obtain
  \begin{align*}
    G_\n(\w, z) &= \frac{(1 + g_\n + g_\n') (1 + H_{\n-1}' + H_{\n-1} )}
          {1 - g_\n' (H_{\n-1}' + H_{\n-1}) + g_\n (H_{\n-1}' + \w H_{\n-1})} \\
    &= \frac{(1 + g_\n + g_\n') (1 + H_{\n-1}' + H_{\n-1} )}
          {1 - (g_\n' + g_\n)(H_{\n-1}' + H_{\n-1}) - g_\n H_{\n-1} (\w - 1) } . 
  \end{align*}
  Eq.~\eqref{E:genp} is obtained by
  substituting
  \begin{align*}
    g_\n &= \frac{ (p_\n z)^{k_\n}} {1 - p_\n z}, \\
    g_\n + g_\n' &= \frac{ p_\n z} {1 - p_\n z}, \\
    H_{\n-1} &=   \frac{\prod_{i=1}^{\n-1} p_i^{k_i}  z^k }
              {(1 - z \sum_{i=1}^{\n-1} p_i  ) \prod_{i=2}^{\n-1} (1-p_iz)},  \\
    H_{\n-1} + H_{\n-1}' &= \frac{ (p_1 + p_2 + \cdots + p_{\n-1}) z}{1 - (p_1 + p_2 + \cdots + p_{\n-1}) z}
  \end{align*}
  into the above equation and using $\sum_{i=1}^\n p_i  = 1$.
\end{proof}
\end{theorem}

The i.i.d. version for independent identical distribution trials is stated in the following Corollary.
\begin{corollary}[for i.i.d. distribution of $(k_1, \dots, k_\n)$]
  The double generating function $\widetilde{G}_\n(\w, z)$ for i.i.d. trials is given by
  \begin{equation*} \label{E:gen}
    \widetilde{G}_\n(\w, z) = \frac{(1-z)^{\n-2}}{(1-z)^{\n-2}(1-\n z) - (w-1) z^{k}},
  \end{equation*}
  \begin{proof}
    Let $p_i = 1/\n$ for $i=1, 2, \dots, \n$ and substitute $z$ by $\n z$ in Eq.~\eqref{E:genp}.
    \end{proof}
\end{corollary}

From Eq.~\eqref{E:genp} of Theorem~\ref{Th:genp} it is evident that for any given $\n$,
the double \gf{} $G_\n(\w, z)$ is a rational function.
It is well known that any property that is described by a rational
\gf{} satisfies the recurrence relation determined by the denominator of the rational
\gf{}.
If the denominator of  $G_\n(\w, z)$ in Eq.~\eqref{E:genp} is written as
\[
  (1-z) \prod_{i=2}^{\n-1} (1-p_iz)   - (w-1) \prod_{i=1}^{\n} p_i^{k_i}   z^{k}
  =  \B_\n(z) - (w-1) \prod_{i=1}^{\n} p_i^{k_i}   z^{k} ,
\]
whose first product is expanded as
\begin{equation} \label{E:expansion}
  \B_\n(z) = (1-z) \prod_{i=2}^{\n-1} (1-p_iz) = 1 - a_1 z - a_2 z^2 - \dots - (-1)^{\n} \prod_{i=2}^{\n-1} p_i z^{\n-1},
\end{equation}
then we have the following Corollary for the recurrence relation satisfied by $\F_{\n,n}(\w)$:
\begin{corollary}[recurrence relation of $\F_{\n,n}(\w)$] \label{C:recur}
  When $0 \le n < k=\sum_{i=1}^{\n} k_i$,
  \begin{equation*} 
    \F_{\n,n}(\w) = 1,
  \end{equation*}
  otherwise
\begin{multline} \label{E:recur}
  \F_{\n,n}(\w) =  a_1 \F_{\n,n-1}(\w)
  + a_2  \F_{\n,n-2}(\w)
  + \dots  
  + (-1)^{\n} \prod_{i=2}^{\n-1} p_i \F_{\n, n-\n+1}(\w) \\
  + (\w-1) \prod_{i=1}^{\n} p_i^{k_i} \F_{\n, n-k}(\w) .
\end{multline}
\end{corollary}
It's straightforward to write down the explicit recurrence relation of $\F_{\n, n}(\w)$
as shown in Eq.~\eqref{E:recur}
for any particular $\n$
using the expansion in Eq.~\eqref{E:expansion}.
For examples, the recurrence relations for $\n=2$, $\n=3$, and $\n=4$ are listed below.
When $\n=2$, $\B_2(z) = 1 - z$, hence $a_1=1$ and $a_i=0$ for $i>1$.
The recurrence for $\F_{2, n}(\w)$ is given by
\[
  \F_{2, n}(\w) = \F_{2, n-1}(\w) + (\w-1) p_1^{k_1}p_2^{k_2} \F_{2, n-k}(\w) , \quad n \ge k.
\]
This relation is the same as the Theorem 3.1 of ~\citep{DAFNIS20101691}
(The $\Phi$ on the LHS of that equation is a misprint, which should be $\phi$).

When $\n=3$,
\[
  \B_3(z) = (1-z)  (1-p_2z) = 1 - (1 + p_2) z + p_2 z^2,
\]
and $\F_{3, n}(\w)$ satisfies the following recurrence relation when $n \ge k$
\[
  \F_{3, n}(\w) = (1 + p_2) \F_{3, n-1}(\w) - p_2 \F_{3, n-2}(\w) + (\w-1)\prod_{i=1}^{3} p_i^{k_i} \F_{3, n-k}(\w) .
\]
When $\n=4$,
\[
  \B_4(z) = (1-z)  (1-p_2z)(1-p_3z) = 1 - (1 + p_2 + p_3) z + (p_2 +p_3 +p_2p_3)z^2 - p_2p_3 z^3,
\]
hence $\F_{4, n}(\w)$ satisfies the following recurrence relation when $n \ge k$
\begin{multline*}
  \F_{4, n}(\w) = (1 + p_2 + p_3) \F_{4, n-1}(\w)
  - (p_2 +p_3 +p_2p_3) \F_{4, n-2}(\w)
  + p_2p_3 \F_{4, n-3}(\w) \\
  + (\w-1)\prod_{i=1}^{4} p_i^{k_i} \F_{4, n-k}(\w) .
\end{multline*}

From Theorem~\ref{Th:genp} the average of $\X_{\n; k_1, \dots, k_\n}^{(n)}$ and other statistics can be obtained.
The following Corollary gives the principal part of the average of $\X_{\n; k_1, \dots, k_\n}^{(n)}$
for a sequence of length $n$. It shows that as the length goes to infinity, the average depends
linearly on the length $n$.
\begin{corollary}[average of $\X_{\n; k_1, \dots, k_\n}^{(n)}$]
  The average of $\X_{\n; k_1, \dots, k_\n}^{(n)}$
for a sequence of length $n$ is given by
  \begin{equation} \label{E:agen}
  E[ \X_{\n; k_1, \dots, k_\n}^{(n)}] = 
    \frac{ \prod_{i=1}^{\n} p_i^{k_i}  (n-k+1)  }{\prod_{i=2}^{\n-1} (1-p_i)} + O(1).
  \end{equation}
  \begin{proof}
    Take the derivative of Eq.~\eqref{E:genp} with respect to $\w$, and set $\w=1$:
\[
  \frac{ \pd G_\n(\w, z)} {\pd \w} \Big |_{\w=1} = \frac{\prod_{i=1}^{\n} p_i^{k_i} z^k}
  {(1-z)^2 \prod_{i=2}^{\n-1} (1-p_iz) } .    
\]
Eq.~\eqref{E:agen} is obtained by extracting the coefficient of $z$ from the above equation.
  \end{proof}
\end{corollary}

\section{Number of sequences containing \texorpdfstring{$(k_1, \dots, k_\n)$}{k1,...,kl} patterns: the i.i.d. trials } \label{S:iid}
In previous sections the generating functions obtained are for non-iid trials
where  probability generating functions
\[
  G_\n(\w, z) = \sum_{n=0}^\infty \sum_{j=0} \PP_{\n,n}( j ) \w^j z^n
\]
are obtained,
where $\PP_{\n,n}( m  )$ is the probability for the sequences of length $n$ to contain $m$ $(k_1, \dots, k_\n)$ patterns.
If we are interested in the \emph{numbers} of sequences that
contain $(k_1, k_2, \dots, k_\n)$ patterns, assuming independent identical distribution (i.i.d.)
for the objects,
the following generating function will be obtained
\[
  \Giid_\n(\w, z) = \sum_{n=0}^\infty \sum_{j=0} \A_{\n,n}( j) \w^j z^n  ,
\]
where $A_{\n,n}(m)$ is the number of sequences with $m$ $(k_1, \dots, k_\n)$ patterns, for trial sequences of length $n$.

Below we list the results for the i.i.d. trials with proofs omitted.
The proofs for the following results are similar to those in the previous sections,
with the generating functions $g_i$ and  $g_i'$ replaced by
\begin{align*}
  g_i' &= \sum_{j=1}^{k_i - 1} z^j =  \frac{z - z^{k_i}}{1-z},  \\
  g_i  &= \sum_{j=k_i}^\infty z^j = \frac{z^{k_i}}{1-z} .
\end{align*}
Alternatively, these results can be derived from the corresponding non-iid results
by substituting $p_i$ by $1/\n$ for $i=1, \dots, \n$, and $z$ by $\n z$.

\begin{proposition}[distribution of $(k_1, k_2)$] \label{P:k1k2}
  The number of sequences of length $n$ with $m$ $(k_1, k_2)$ patterns
  is given by $\A_{2,n}( m) = [\w^m z^n] \Giid_2(w, z) $,
where
\begin{equation} \label{E:k1k2}
  \Giid_2(\w, z) = \frac{1}{1-2z - (\w-1) z^{k_1+k_2}} .
\end{equation}
\end{proposition}

\begin{lemma}[$\n=2$, with  pattern $(k_1, k_2)$ at the right end] \label{L:k1k2p}
  The generating functions for the number of sequences with
  pattern $(k_1, k_2)$ at the right end
  and its complement
  are given by
\begin{align*} 
  \widetilde{H}_2(z)  &= \frac{z^{k_1+k_2}}{(1-z)(1-2z)} , \\
  \widetilde{H}_2'(z) &= \frac{2z(1-z) - z^{k_1+k_2}}{(1-z)(1-2z)} .
\end{align*}
\end{lemma}

\begin{proposition}[distribution of $(k_1, k_2, k_3)$] \label{E:G3p}
The number of sequences of length $n$ with $m$ $(k_1, k_2, k_3)$ patterns
  is given by $\A_{3,n}(m) = [\w^m z^n] \Giid_3(\w, z) $,  
where
\begin{equation} \label{E:k1k2k3}
    \Giid_3(\w, z) = \frac{1-z}{(1-z)(1-3z) - (w-1) z^{k}},
\end{equation}
with $k=k_1+k_2+k_3$.
\end{proposition}
Note the symmetry for $k_1, k_2$, and $k_3$ in Eq.~\eqref{E:k1k2k3} (as well as Eq.~\eqref{E:k1k2}).
This symmetry is broken in non-iid cases,
where non-identical probabilities for each object are used (Eq.~\eqref{E:k1k2k3p} for $\n=3$
and Eq.~\eqref{E:genp} for arbitrary $\n$).

As a numerical example for $\n=3$, the following is the generating function
$\Fiid_{3, 17}(\w) =  \sum_j A_{3,n}( j) \w^j$ for $n=17, k_1=2, k_2=2, k_3=3$:
\[
  \Fiid_{3, 17}(\w) =  128210550 + 929204 \w + 409 \w^2 .
\]
The \gf{} $\Fiid_{3, 17}(\w)$ shows that, among $3^{17}$ sequences of length $17$,
there are $128210550$, $929204$, and $409$ sequences with $0$, $1$, and $2$
$(k_1, k_2, k_3)$ patterns, respectively.

%
%

\begin{corollary}[average of $\X_{3; k_1,k_2, k_3}^{(n)}$]
  The average of $\X_{3; k_1,k_2, k_3}^{(n)}$ is given by
\begin{equation*} \label{E:a3}
 E[ \X_{3; k_1,k_2, k_3}^{(n)} ] = \frac{1}{3^n} \left[ \frac{1}{4} + \frac{1}{4} (1+2n-2k)3^{n-k+1}   \right] ,
\end{equation*}
where $k=k_1 + k_2 + k_3$.
\end{corollary}

\section{Concluding Remarks} \label{S:rem}

The interaction factors $\w_{ij}$ in Eq.~\eqref{E:M} naturally link two adjacent blocks together,
so the method described in Section~\ref{S:method} can easily be used to study compound patterns that involve two sub-patterns,
such as $(k_1, k_2)$ as shown in Section~\ref{S:2}.
When the method is applied iteratively,  with the help of some special patterns
(such as those shown in Section~\ref{S:2p}),
compound patterns that involve multiple sub-patterns can also be handled by the method.
In this paper we only apply the method to one particular compound pattern.
Apparently, other compound patterns that are composed of a set of multiple sub-patterns
can also be investigated by the method in a similar manner.
We hope that the method can be used as a basis for further research
on various other complicated patterns.


 \newcommand{\noop}[1]{}

\end{document}